\documentclass[11pt,english]{smfart}

\usepackage{amsmath,leftidx,amsthm,sfheaders}
\usepackage[all]{xy}
\usepackage{xspace,smfthm}
\usepackage[psamsfonts]{amssymb}
\usepackage[latin1]{inputenc}
\usepackage{graphicx,color}
\usepackage{hyperref,fancyhdr,dutchcal,appendix}
\usepackage{xcolor}

\newtheorem*{remark}{\bf Remark}

\newtheorem{theorem}{\bf Theorem}

\newtheorem{lemma}[theorem]{\bf Lemma}

\def\C{{\mathbb C}}
\def\N{{\mathbb N}}
\def\R{{\mathbb R}}

\def\D{\mathbb{D}}

\def\p{\mathbb{P}}

\def\and{{\quad\text{and}\quad}}

\title{The Geometric Dynamical Northcott Property in the quadratic family}
\author{Thomas Gauthier}
\address{Laboratoire de Math\'ematiques d'Orsay, B\^atiment 307, Universit\'e Paris-Saclay, 91405 Orsay Cedex, France}
\email{thomas.gauthier1@universite-paris-saclay.fr}
\author{Gabriel Vigny}
\address{LAMFA, Universit\'e de Picardie Jules Verne, 33 rue Saint-Leu, 80039 AMIENS Cedex 1, FRANCE}
\email{gabriel.vigny@u-picardie.fr}
\thanks{The first author is partially supported by the Institut Universitaire de France.}
\thanks{The second author is partially supported by the ANR QuaSiDy /ANR-21-CE40-0016.}

\begin{document}
\maketitle

\begin{abstract}
The aim of this note is to give a proof of Theorem A from \cite{GV_Northcott} in the simpler case of the quadratic family; being in dimension $1$ in both the dynamical space and the parameter space, and having a simple and explicit parametrization of the family allow to simplify the proof and, we hope, make the ideas more apparent.
\end{abstract}

Consider the quadratic family 
\[\begin{cases}
f:&\C\times \C \to \C \times \C \\
& \, \ (z,\lambda) \mapsto (z^2+\lambda, \lambda).
\end{cases} \]
For each $\lambda$, $f_\lambda(z):=z^2+\lambda $ defines a polynomial map on $\C$ whose \emph{filled Julia set} $K_\lambda$ is 
\[ K_\lambda := \{z\in \C, \ \limsup |f^n_\lambda(z)| <\infty \}.\]
It is a (perfect) compact set in $\C$ and one can similarly consider
\[ \mathcal{K}:= \{(z,\lambda)\in \C^2, \ \limsup |f^n_\lambda(z)| <\infty \}=\bigcup_{\lambda\in\C} K_\lambda \times \{\lambda\}.\]
In this situation, the Geometric Dynamical Northcott Property can be rephrased as 
\begin{theorem}[\cite{benedetto, Baker-functionfield, demarco}]\label{quadratic_case}
The only algebraic curves of $\C^2$ contained in $\mathcal{K}$ are preperiodic. 
\end{theorem}
We gave a generalization of this result for families of polarized endomorphisms over a projective variety in \cite{GV_Northcott} (see also \cite{Chatzidakis-Hrushovski}) though the proof is quite involved due to the necessity to deal with singular projective varieties, to take normalizations of several such projective varieties, and to deal with possible isotrivial subvarieties \dots  The purpose of this note is to give the proof in the simplest possible case for a reader who would like to have a good understanding of the main ideas used in \cite{GV_Northcott}. 

\medskip

The \emph{Green current} $T_f$ of $f$ is defined as $T_f:=\lim_{n\to+\infty} 2^{-n} (f^n)^* (\pi_1^*(\omega))$ where $\pi_i:\C\times \C \to \C$ is the projection on the $i$-th factor, and $\omega$ is the Fubini-Study form of $\p^1$. The current $T_f$ is a positive closed current with continuous potential $G:= \lim 2^{-n} \log\max (|f^n(z,\lambda)|, 1)$ (i.e.\ $dd^c G =T_f$ in $\C^2$) and its slice with any vertical line $\{\lambda\}\times \C$ is the \emph{Green measure $\mu_\lambda$} of $f_\lambda$: $\mu_\lambda$ is the unique ergodic measure of maximal entropy $\log 2$, and it satisfies the equidistribution property
\[\lim_{n\to\infty} \frac{1}{2^n} \sum_{z, \, f_\lambda^n(z)=z} \delta_z = \mu_{\lambda}\]
which is first due to \cite{brolin} in that case. Moreover, we have $\mathcal{K}=\{G=0\}$.
\begin{lemma}
	Let $Z \subset \C^2$ be an algebraic curve. If $Z \subset \mathcal{K}$ then $T_f \wedge [Z] =0$.
\end{lemma}
\begin{proof}
	As $G\equiv0$ on $Z$, we have $G\cdot[Z]=0$ in the sense of currents, so that we find $T_f\wedge [Z]=dd^c(G\cdot[Z])= 0$.
\end{proof}
In the following, $\deg(Z)$ is computed with respect to the ample line bundle $\pi_1^*(\mathcal{O}(1))\oplus \pi_2^*(\mathcal{O}(1))$ (in other words, we count the sum of the number of intersections of $Z$ with a generic vertical line and a generic horizontal line).
\begin{lemma}
	Let $Z \subset \C^2$ be an algebraic curve with no vertical component, then $T_f \wedge [Z] =0$ if and only if $\deg(f^n(Z))=O(1)$.
\end{lemma}
\begin{proof}
	Recall that $G$ is defined by $\lim_n 2^{-n}\log^+|f^n_\lambda(z)|$.  From the inequality $\log^+|a+b|\leq \log^+|a|+\log^+|b|+\log 2$, we easily have
	\[ |\log^+|f_\lambda(z)| -2 \log^+|z|| \leq \log^+|\lambda|+\log 2. \]
	Hence, $| 2^{-k} \log^+|f^k_\lambda(z)| -2^{-k+1} \log^+|f_\lambda^{k-1}(z)| \leq 2^{-k}(\log^+|\lambda|+\log 2)$ and summing from $n+1$ to $\infty$, we deduce that
	\begin{align}\label{estimee_quadratique}
	\forall \lambda,z, \ \left| G(z,\lambda) - 2^{-n}\log^+| f^n_\lambda(z)|\right| &\leq \sum_{n+1}^\infty   | 2^{-k} \log^+|f^k_\lambda(z)| -2^{-k+1} \log^+|f_\lambda^{k-1}(z)| \nonumber \\
	&\leq \frac{\log^+|\lambda|+\log 2}{2^{n-1}}. 
	\end{align}
	Let $A\gg 1$ and consider the cut-off function on $\C$
	\[\phi_A:= \frac{ \log \max (|\lambda|,e^{2A})- \log \max (|\lambda|,e^{A})}{A}.\]
	Then, $\phi_A$ is equal to $1$ on the disk $\D(e^A)$ centered at $0$ of radius $e^A$, is zero outside $\D(e^{2A})$ and $dd^c \phi_A = \frac{\lambda_{S_{2A}}- \lambda_{S_{A}}}{A}$ where $\lambda_{S_A}$ (resp. $\lambda_{S_{2A}}$) is the unit (of mass $1$) Lebesgue measure  on the circle of radius $e^A$ (resp. radius $e^{2A}$).
	
	Take an algebraic curve $Z \subset \C^2$ with no vertical components.	We compute, using the fact that $f$ acts trivially on the variable $\lambda$ 
	\begin{align*}
	\int_{\C^2}  \phi_A f^n_*([Z]) \wedge (\pi_2^*(\omega)+\pi_1^*(\omega))&=\int_{\C^2}  \phi_A [Z] \wedge (\pi_2^*(\omega)+ (f^n)^*(\pi_1^*(\omega))) \\
	&= \int_{\C^2}  \phi_A [Z] \wedge \pi_2^*(\omega) + \int_{\C^2}  \phi_A [Z] \wedge 2^nT_f\\
	& \quad + 2^n \int_{\C^2}  \phi_A [Z] \wedge (2^{-n} (f^n)^*(\pi_1^*(\omega)) - T_f). 
	\end{align*}
Now, since we have $2^{-n} (f^n)^*(\pi_1^*(\omega)) - T_f= dd^c( 2^{-n}\log^+|f_\lambda(z)|-G)$, by Stokes, \eqref{estimee_quadratique} and Bézout's theorem, we find
	\begin{align*}
	I_n: &=    \left|2^n \int_{\C^2}  \phi_A [Z] \wedge (2^{-n} (f^n)^*(\pi_1^*(\omega)) - T_f) \right|\\
	&= \left|2^n \int_{\C^2} (2^{-n}\log^+|f_\lambda(z)|-G)  dd^c \phi_A \wedge [Z] \right| \\
	&\leq 2 \int_{\C^2}  (\log^+|\lambda|+\log 2) \frac{\pi_1^*(\lambda_A)}{A} \wedge [Z] +
	 2\int_{\C^2} (\log^+|\lambda|+\log 2)  \frac{\pi_1^*(\lambda_{2A})}{A} \wedge [Z]  \\
	&\leq  2\int_{\C^2} (A+\log2) \frac{\pi_1^*(\lambda_{S_A})}{A} \wedge [Z] +2\int_{\C^2} (2A+\log 2) \frac{\pi_1^*(\lambda_{S_{2A}})}{A} \wedge [Z] \\
	&  \leq C \deg(Z)
	\end{align*}
	where $C$ is a constant that depends neither on $A$, nor on $Z$, nor on $n$. So, up to taking a larger $C$, and by letting $A\to \infty$, the degree $\deg f^n(Z)$
	\[ \deg f^n(Z):=	\int_{\C^2} f^n_*([Z]) \wedge (\pi_1^*(\omega)+\pi_2^*(\omega)) .\]
satisfies
	\[ \left|\deg f^n(Z)-  2^n\int_{\C^2} [Z]\wedge T_f \right|\leq C \deg(Z)\]
	which ends the proof.
\end{proof}

\begin{remark} \normalfont
We may see the family $f$ as a dynamical system $\mathsf{f}$ on $\p^1(\mathbf{K})$, where $\mathbf{K}$ is the field of complex rational functions. Any point $\mathsf{z}\in\p^1(\mathbf{K})$ corresponds to a rational function $z:\p^1\to \p^1$ and the dynamical height function $\widehat{h}_\mathsf{f}:\p^1(\overline{\mathbf{K}})\to\R_+$
is defined on $\p^1(\mathbf{K})$ by
\[\widehat{h}_\mathsf{f}(\mathsf{z}):=\lim_{n\to\infty}\frac{1}{2^n}\deg(\mathsf{z}_n),\]
where $z_n:\p^1\to\p^1$ corresponds to $\mathsf{z}_n$ and is defined by $z_n(\lambda):=f_{\lambda}^n(z(\lambda))$.
 In particular, the lemma gives
\[\widehat{h}_\mathsf{f}(\mathsf{z})=\int_{\C^2}T_f\wedge [Z],\]
when $Z$ is the graph of $\mathsf{z}\in \p^1(\mathbf{K})$. In other words, we proved that
\[\left|\deg(f^n(Z))-2^n \widehat{h}_f(\mathsf{z})\right|\leq C\deg(Z).\]
\end{remark}

\begin{lemma}\label{lemma4}
	Let $Z \subset \C^2$ be an algebraic curve with no vertical components, if $\deg(f^n(Z))=O(1)$ then $Z$ is preperiodic.
\end{lemma}
\begin{proof}
	Let $\mathcal{A}_D$ denote the set of algebraic sets $Z$ of degree $\leq D$ where, from now on, we compute the degree in $\p^2$ (having bounded degree in $\p^2$ or in $\p^1\times \p^1$ is equivalent). Such $Z$ is defined by some equation:
	\[\sum_{i+j\leq D} a_{i,j} \lambda^i z^j =0 \]
	where the $(a_{i,j})\in \p^{\frac{(D+1)(D+2)}{2}}(\C)$. So $\mathcal{A}_D$ is an algebraic variety (notice that $a_{i,j}=0$ for $(i,j)\neq(0,0)$ corresponds to the line at infinity).  Usually, computing the direct image of an analytic set cannot be done explicitly, but here it is possible:
	write $\sum_{i+j\leq D} a_{i,j} \lambda^i z^j =\sum_{i+2j\leq D} a'_{i,j} \lambda^i (z^2+\lambda)^j + z \sum_{i+2j\leq D-1} a''_{i,j} \lambda^i (z^2+\lambda)^j $ for some suitable $a'_{i,j}$, $a''_{i,j}$ depending linearly on the $(a_{i,j})$. Then,
	\begin{align*} \sum_{i+j\leq D} a_{i,j} \lambda^i z^j =0 \ &\iff \ \sum_{i+2j\leq D} a'_{i,j} \lambda^i (z^2+\lambda)^j =- z \sum_{i+2j\leq D-1} a''_{i,j} \lambda^i (z^2+\lambda)^j  \\ 
	& \iff  \sum_{i+2j\leq D} a'_{i,j} \lambda^i f_\lambda(z)^j =- z \sum_{i+2j\leq D-1} a''_{i,j} \lambda^i f_\lambda(z).
	\end{align*}
	Take the square (this does not add points because the set $-Z:=\{( z, \lambda),(-z,\lambda)\in Z\} $ has the same image than $Z$) and compute:
	\begin{align*}  \left(\sum_{i+2j\leq D} a'_{i,j} \lambda^i f_\lambda(z)^j\right)^2 =  & (z^2+\lambda) \left(\sum_{i+2j\leq D-1} a''_{i,j} \lambda^i (f_\lambda(z))^j\right)^2\\
	& \quad - \lambda\left(\sum_{i+2j\leq D-1} a''_{i,j} \lambda^i (f_\lambda(z))^j\right)^2.
	\end{align*}
	So we recognize that $f(Z)$ is given by the equation
	\[ \left(\sum_{i+2j\leq D} a'_{i,j} \lambda^i z^j\right)^2 = z \left(\sum_{i+2j\leq D-1} a''_{i,j} \lambda^i z^j\right)^2- \lambda\left(\sum_{i+2j\leq D-1} a''_{i,j} \lambda^i z^j\right)^2.   \]
In particular, the application that sends $Z$ to $f(Z)$ is a morphism (a priori from $\mathcal{A}_D$ to $\mathcal{A}_{dD}$) and the condition that $\deg(f(Z)) \leq D$ is an algebraic condition. Intersecting, 
	\[ \{Z\in \mathcal{A}_D, \ \ \forall n, \  f^n(Z)\in \mathcal{A}_D\}   \]
	is a subvariety of $\mathcal{A}_D$.
	
	We now start with a horizontal irreducible algebraic curve $Z \in \mathcal{A}_D$ with $\deg(f^n(Z))\leq D$ for all $n$ and consider the Zariski closure $\mathcal{Z}$ of $\{ f^n(Z), \ n\in \N\} $ in $ \mathcal{A}_D$. Observe that $f$ induces an action $\mathfrak{f}: \mathcal{Z} \to \mathcal{Z}$. In particular, there is an irreducible component $\mathcal{Z}_1$ of $\mathcal{Z}$ with $f^k(\mathcal{Z}_1)=\mathcal{Z}_1$ for some $k$. Without loss of generality, we may assume $k=1$ and $\mathcal{Z}_1=\mathcal{Z}$ in the rest of the proof. Furthermore, a generic element of $\mathcal{Z}$ is irreducible by construction. If $\mathcal{Z}$ has dimension $0$, then it is finite and $Z$ is preperiodic so we are done. 
	
\medskip

	Assume by contradiction that $\dim(\mathcal{Z})\geq 1$.
	Consider the set
	\[\widehat{\mathcal{Z}}:=\{(Z,z,\lambda) \in \mathcal{Z} \times \C^2, \,  (z,\lambda)\in Z \}. \]
	Then, $\widehat{\mathcal{Z}}$ is a subvariety of $\mathcal{A}_D\times \C^2$ and our hypothesis implies that the canonical projection $\Pi: \widehat{\mathcal{Z}} \to \C^2$ onto the second factor, i.e. defined by $\Pi(Z,z,\lambda)=(z,\lambda)$ is dominant (if not, its image is a strict algebraic subvariety of $\C^2$ which would contradict our assumption that $\dim(\mathcal{Z})\geq 1$). 

	In particular, there is a non-empty Zariski open set $W\subset\C^2$ such that for any $(z,\lambda) \in W$, there exists an irreducible $Z_0 \in \mathcal{Z}$ such that  $(z,\lambda)\in Z_0$ (if not, we have that for infinitely many $\lambda$, a set in $\mathcal{Z}$ has to pass through finitely many points and so $\mathcal{Z}$ is finite). 
	\begin{lemma}\label{transversality_quadratic}
		Let $Z_0 \in \mathcal{Z}$ be irreducible such that $(z_0,\lambda_0)\in Z_0$ where $z_0$ is a repelling periodic point of $f_{\lambda_0}$ of period $k$. Then  $Z_0 = \{(z,\lambda), \ f^k_\lambda(z)=z\}$.
	\end{lemma}	
	Take the lemma for granted and continue the proof. As repelling periodic points of a polynomial are Zariski dense in $\C$, this implies that generically in $(z,\lambda)$, one and only one $Z\in \mathcal{Z}$ passes through $(z,\lambda)$. In particular, $\dim(\mathcal{Z})= 1$ and the projection $\Pi$ is finite-to-1 onto its image. Up to taking a base change $\mathcal{B}\to \C$, we can assume that a given $Z_0\in \mathcal{Z}$ is an analytic graph hence, every $Z\in \mathcal{Z}$ is a graph. In particular, for every $(z,\lambda)$ in a Zariski dense open set of $\C \times\mathbb{B}$, there exists a unique $Z \in \mathcal{Z}$ such that $(z,\lambda)\in Z$. Fix two generic close $\lambda_0, \lambda_1$, outside the Mandelbrot set (in particular, periodic points can be followed holomorphically). Let us denote by $\Phi:\C \to \C $ the application that sends $z$ to the intersection of the leaf that contains $z$ at $\lambda_0$  with $\C  \times \{\lambda_1\}$.    
Then $f_{\lambda_1}(\Phi(z))=\Phi(f_{\lambda_0}(z))$ for every periodic point, hence for all points by Zariski density. In particular, we have a holomorphic (hence affine) conjugacy between $f_{\lambda_1}$ and $f_{\lambda_0}$. This is absurd, as $f_{\lambda_1}$ and $f_{\lambda_0}$ are holomorphically conjugate if and only if $\lambda_1=\lambda_0$.
\end{proof}

We now prove Lemma~\ref{transversality_quadratic}.
\begin{proof}[Proof of Lemma~\ref{transversality_quadratic}]
	Let us fix such $\lambda_0$ and $z_0$ a repelling periodic point of period $k$ of $f_{\lambda_0}$; we can follow that periodic point holomorphically by $\lambda\mapsto y(\lambda)$. Let $Z_0 \in \mathcal{Z}$, irreducible, such that  $(z_0,\lambda_0)\in Z_0$. Assume that the intersection $Z_0 \cap \{(y(\lambda), \lambda)\}$ is proper at $(z_0,\lambda_0)$. Up to reparametrizing, we can 
	follow locally a branch of $Z_0$ that contains  $(z_0,\lambda_0)$ through a graph $\lambda\mapsto z(\lambda)$ and our hypothesis means that for every $\lambda\neq \lambda_0$ in a neighborhood of $\lambda_0$, $y(\lambda)\neq z(\lambda)$. 
	
	By hypothesis, the Green function $\lambda\mapsto G(z(\lambda),\lambda)$ is harmonic and it admits a minimum at $(\lambda_0)$ so it is identically $0$. In particular, $(z(\lambda), \lambda)\in \mathcal{K}$ so the sequence $(\lambda\mapsto  f^n_\lambda(z(\lambda)))_n$ is normal. 
	In particular, for $\varepsilon>0$ small enough, we can find $K>1$ and $\delta>0$ small enough so that, for every $n$, $|f_\lambda^{kn}(z(\lambda))-y(\lambda)|< \varepsilon$ for $|\lambda-\lambda_0|<\delta$ (indeed, this is a normal sequence that is $0$ at $\lambda_0$) and $|f_\lambda^{k}(z)-y(\lambda)| \geq K |z-y(\lambda)|$ for $|z-y(\lambda)|<\varepsilon$. By iteration, we deduce $|f_\lambda^{kn}(z(\lambda))-y(\lambda)| \geq K^n |z(\lambda)-y(\lambda)|$ as $y(\lambda)$ is repelling, a contradiction. 
	
	In particular, by irreducibility, $Z_0\subset \{(z,\lambda), \ f^k_\lambda(z)=z\}$. Finally, as $\{(z,\lambda), \ f^k_\lambda(z)=z\}$ is irreducible (e.g. \cite{BuffLei}), we have the equality.
\end{proof}	
Now, Theorem~\ref{quadratic_case} then follows from the three above lemmas, since the assumption $Z\subset\mathcal{K}$ implies $Z$ has no vertical components as $\mathcal{K}\cap \C \times \{\lambda\}$ is compact in $ \C \times\{\lambda\}$.

\begin{remark}
	\normalfont In the particular case we are in, we can give a very short alternate argument of the end of the proof of Lemma~\ref{lemma4} using \cite{BuffLei}: above a Zariski generic parameter $\lambda_0$, for a Zariski dense subset of periodic points $z$ for $f_{\lambda_0}$, the set $Z \in \mathcal{Z}$ that passes through $(z,\lambda_0)$ is of the form $\{(z,\lambda), \ f^k_\lambda(z)=z\}$. But they are only finitely many such algebraic sets of degree $\leq D$. 
\end{remark}

\bibliographystyle{short}
\bibliography{biblio}

\begin{thebibliography}{GV}

\bibitem[Ba]{Baker-functionfield}
Matthew Baker.
\newblock A finiteness theorem for canonical heights attached to rational maps
  over function fields.
\newblock {\em J. Reine Angew. Math.}, 626:205--233, 2009.

\bibitem[Be]{benedetto}
Robert~L. Benedetto.
\newblock Heights and preperiodic points of polynomials over function fields.
\newblock {\em Int. Math. Res. Not.}, (62):3855--3866, 2005.

\bibitem[Br]{brolin}
Hans Brolin.
\newblock Invariant sets under iteration of rational functions.
\newblock {\em Ark. Mat.}, 6:103--144 (1965), 1965.

\bibitem[BL]{BuffLei}
Xavier Buff and Tan Lei.
\newblock The quadratic dynatomic curves are smooth and irreducible.
\newblock In {\em Frontiers in complex dynamics. In celebration of John
  Milnor's 80th birthday. Based on a conference, Banff, Canada, February 2011},
  pages 49--72. Princeton, NJ: Princeton University Press, 2014.

\bibitem[CH]{Chatzidakis-Hrushovski}
Zo\'{e} Chatzidakis and Ehud Hrushovski.
\newblock Difference fields and descent in algebraic dynamics. {I}.
\newblock {\em J. Inst. Math. Jussieu}, 7(4):653--686, 2008.

\bibitem[D]{demarco}
Laura DeMarco.
\newblock Bifurcations, intersections, and heights.
\newblock {\em Algebra Number Theory}, 10(5):1031--1056, 2016.

\bibitem[GV]{GV_Northcott}
Thomas Gauthier and Gabriel Vigny.
\newblock The geometric dynamical northcott and bogomolov properties.

\end{thebibliography}
\end{document}